\documentclass{amsart}
\usepackage{amsmath,amsthm}
\usepackage{amsfonts,amssymb}
\usepackage{graphicx} 
\usepackage{multirow}

\theoremstyle{plain}
\newtheorem{thm}{Theorem}[section] 
\newtheorem{lem}[thm]{Lemma} 
\newtheorem{prop}[thm]{Proposition}

\theoremstyle{definition}

\theoremstyle{remark}

\def\dive{\operatorname{div}}

\begin{document}
 
\title[A fully diagonalized spectral method on the unit ball]
{A fully diagonalized spectral method on the unit ball}

\author[M. A. Pi\~{n}ar]{Miguel A. Pi\~{n}ar}
\address[M. A. Pi\~{n}ar]{{Departamento de Matem\'atica Aplicada\\
Universidad de Granada\\
18071 Granada, Spain}}\email{mpinar@ugr.es}

\thanks{The author thanks research project PID2023.149117NB.I00  funded  by MICIU/AEI/10.13039/501100011033 and by ``ERDF A Way of making Europe”}

\date{\today}
\keywords{Sobolev orthogonal polynomials, Dirichlet problem, unit ball}
\subjclass[2000]{33C50, 42C10}

\begin{abstract} 
Our main objective in this work is to show how Sobolev orthogonal polynomials emerge as a useful tool within the framework of spectral methods for boundary-value problems. The solution of a boundary-value problem for a stationary Schr\"odinger equation on the unit ball can be studied from a variational perspective. In this variational formulation, a Sobolev inner product naturally arises. As test functions, we consider the linear space of the polynomials satisfying the boundary conditions on the sphere, and a basis of mutually orthogonal polynomials with respect to the Sobolev inner product is provided. The basis of the proposed method is given in terms of spherical harmonics and univariate Sobolev orthogonal polynomials. The connection formula between these Sobolev orthogonal polynomials and the classical orthogonal polynomials on the ball is established. Consequently, the Sobolev Fourier coefficients of a function satisfying the boundary value problem are recursively derived. Finally, one numerical experiment is presented.
\end{abstract}

\maketitle

\section{Introduction}
\setcounter{equation}{0}

The purpose of this paper is to numerically solve a Dirichlet boundary value problem on the 
unit ball \(\mathbb{B}^d = \{x: \|x\|^2\le 1\}\), where \(\|x\|^2 = x_1^2 + \ldots +x_d^2\) defines the usual Euclidean norm in \(\mathbb{R}^d\). In particular, we will consider the solution of the elliptic partial differential equation
\begin{equation} \label{pde-kappa}
 -  \Delta u  + \lambda(1-\|x\|^2)^{\kappa}u = f(x) \qquad x \in {\mathbb B}^d,
\end{equation}
with the Dirichlet boundary condition
\begin{equation} \label{boundary}
 u(x)= 0 \qquad x \in \partial\mathbb{B}^d = {\mathbb S}^{d-1},
\end{equation}
where \(\lambda>0\), \(\kappa\) is a nonnegative integer, and \(f(x)\) is a several times continuously differentiable function on \({\mathbb B}^d\).

For the Dirichlet problem \eqref{pde-kappa}-\eqref{boundary} the standard variational formulation reads as follows: Find \(u \in H^1_0({\mathbb B}^d)\) such that
\begin{equation} \label{var_form}
\int_{{\mathbb B}^d} \left[\nabla u \cdot\nabla v + \lambda\, u(x)\,v(x)\,(1-\|x\|^2)^{\kappa}\right]\,dx = \int_{{\mathbb B}^d} f(x)\, v(x)\,dx, \quad \forall 
v \in H^1_0({\mathbb B}^d)
\end{equation}
Introduce the bilinear form  
\begin{equation} \label{bil_form_A}
\mathcal{A}(v,w) = \int_{{\mathbb B}^d} \left[\nabla v \cdot\nabla w + \lambda\, v(x)\,w(x)\,(1-\|x\|^2)^{\kappa}\right]\,dx  \quad \forall 
v, w \in H^1_0({\mathbb B}^d),
\end{equation}
and the bounded linear functional
\[
\mathcal{L}(v) = \int_{{\mathbb B}^d} f(x)\, v(x)\,dx, \quad \forall 
v \in H^1_0({\mathbb B}^d).
\]
Therefore, the variational problem \eqref{var_form} can be expressed as follows: find \(u \in H^1_0({\mathbb B}^d)\) for which \(\mathcal{A}(u,v) = \mathcal{L}(v) \quad \forall v \in H^1_0({\mathbb B}^d).\) 

It is straightforward to show \(\mathcal{A}\) is coercive and bounded. Assuming some mild conditions on \(f(x)\) the Lax-Milgram Theorem implies the existence of a unique solution \(u\) to \eqref{var_form}. 

Let $\Pi_n^d$ be the space of polynomials of total degree at most $n$ in $d$ variables. 
To implement a Galerkin method for obtaining an approximate solution to \eqref{var_form} we consider as our test functions the linear space of polynomials satisfying the boundary conditions at the sphere \({\mathbb S}^{d-1}\)
\begin{equation} \label{space_Rn}
\mathcal{R}_n^d = \left\{(1-\|x\|^2) p(x): p(x) \in \Pi_n^d\right\}
\end{equation}
and the problem can be stated as follows: find \(u_n \in \mathcal{R}_n^d\) for which \(\mathcal{A}(u,v) = \mathcal{L}(v) \quad \forall v \in \mathcal{R}_n^d\). Again, the Lax-Milgram Theorem implies the existence of a unique \(u_n\) for all \(n\).
 
There exist many different methods to construct the approximate solution \(u_n\). When the approximation in $L^2$ norm is most relevant, orthogonal functions appear naturally. In this context, classical orthogonal polynomials on the unit ball have been used in spectral methods for solving partial differential equations \cite{ACH1, ACH2, ACH3, LX, STW}. One usual limitation in this technique is that the discretization of partial differential
equations by spectral methods generally leads to the solution of large systems of linear equations involving full matrices, which are usually badly conditioned. In some cases, for \(d=2\) and \(d=3\), classical orthogonal polynomials have been used to devise sparse spectral methods (involving banded matrices, as opposed to full matrices) on the unit ball, see \cite{vas15,vas18}. 

Our approach to this problem differs from the usual one. Realizing that
the bilinear form \(\mathcal{A}\) defines an inner product in \(H^1_0({\mathbb B}^d)\), we can benefit from the variational formulation by constructing a basis of \(\mathcal{R}_n^d\) constituted by polynomials that are orthogonal with respect to \(\mathcal{A}\). 

Orthogonal polynomials on the unit ball with respect to an inner product that involves the integral of $\nabla f \cdot \nabla g$ are usually called \textbf{multivariate Sobolev orthogonal polynomials} and have been studied in a few cases \cite{PPX, PX08, Xu08}. The study in this direction was 
initiated in \cite{Xu08}, where the author considered the inner product
\begin{equation} \label{eq:first-ip}
  \langle f,g \rangle_{I} : = {\frac{\lambda}{\omega_d}} \int_{{\mathbb B}^{d}} \nabla f(x) \cdot \nabla g(x) dx + {\frac{1}{\omega_d}} \int_{{\mathbb S}^{d-1}} f(x) g(x) d\sigma, \quad \lambda  >0, 
\end{equation}
and another one that has the last term on the right-hand side of $\langle f, g\rangle_I$ replaced by ${f(0)g(0)}$. The orthogonal polynomials with respect to the inner product $\langle f, g \rangle_I$ have the distinction that  they are eigenfunctions of a second-order partial differential operator (see \cite{PX08}). The case \(\kappa = 0\) in \eqref{bil_form_A} was considered in \cite{PPX}.

\bigskip

However, to guarantee that our approximate solution satisfies the boundary condition \eqref{boundary} in this work we consider the Sobolev inner product defined in \(\Pi^d\) by means of the following expression: 
\begin{align} 
\langle  u, v\rangle_{S} &:= \mathcal{A} \left((1-\|x\|^2) u(x), (1-\|x\|^2)v(x)\right) =  \lambda\, \int_{{\mathbb B}^{d}} (1-\|x\|^2)^{\kappa+2} u(x) v(x) dx \nonumber  \\ 
  &+ \int_{{\mathbb B}^{d}} \nabla ((1-\|x\|^2)u(x))\,\cdot \nabla ((1-\|x\|^2)v(x)) dx. \label{Sob_ip}
\end{align}
It is easy to see that \(\langle  \cdot, \cdot\rangle_{S}\) defines an inner product in \(\Pi^d\). 

Let us denote by \(\mathcal{V}_n^d(S)\) the space of orthogonal polynomials of degree exactly \(n\) with respect to the inner product \eqref{Sob_ip}. Our first objective is to construct a basis for this space. To take full advantage of the central symmetry of \eqref{Sob_ip} we will reproduce the structure of the classical orthogonal polynomials on the ball, writing the polynomials in spherical polar coordinates as the product of spherical harmonics and a radial part connected to the Jacobi polynomials. 
Next, we will show how these Sobolev orthogonal polynomials appear as a useful tool in the framework of spectral methods for boundary value problems, since they will allow us to obtain easily the coefficients in the Fourier--Sobolev expansion of the solution to the Dirichlet problem \eqref{pde-kappa}-\eqref{boundary},
avoiding the need of solving large linear systems of equations.

\bigskip

The paper is organized as follows. In the next section, we recall the background materials on spherical harmonics and orthogonal polynomials on the unit ball.  In Section 3, Sobolev orthogonal polynomials with respect to the inner product \eqref{Sob_ip} are obtained in spherical polar coordinates. The connection formulas for the radial parts of these polynomials with the Jacobi polynomials are stated and, as a consequence, in Section 4, the Fourier--Sobolev  coefficients of the solution to the Dirichlet problem \eqref{pde-kappa}-\eqref{boundary} are deduced in a recursive way. Finally, a numerical experiment is shown.

\section{Spherical harmonics and orthogonal polynomials on the unit ball}
\setcounter{equation}{0}

In this section, following \cite{FPP}, we describe the background materials about orthogonal polynomials and spherical harmonics. The first subsection collects the properties of the Jacobi polynomials that are needed later. The second subsection presents the basic results on spherical harmonics and classical orthogonal polynomials on the unit ball. 

\subsection{Classical Jacobi polynomials}
We present the properties of the classical Jacobi polynomials $P_n^{(\alpha,\beta)}(t)$ that we need subsequently. Most of these properties can be found in \cite{Sz}. The Jacobi polynomial $P_n^{(\alpha, \beta)}(t)$ is normalized as follows: 
\begin{equation} \label{jac-norm}
P_n^{(\alpha,\beta)}(1) = \binom{n+\alpha}{n} = \frac{(\alpha+1)_n}{n!}, \quad n \in \mathbb{N}_0.
\end{equation}
For $\alpha, \beta > -1$, these 
polynomials are orthogonal with respect to the Jacobi weight function
$$
   w_{\alpha,\beta}(t) := (1-t)^\alpha(1+t)^{\beta}, \qquad  -1< x < 1,
$$
and their $L^2$ norms in $L^2(w_{\alpha,\beta}, [-1,1])$ are given by 
\begin{align}\label{normJ}
h_n^{(\alpha,\beta)} := & \int_{-1}^1 P_{n}^{(\alpha, \beta)}(t)^2 \, w_{\alpha,\beta}(t)\,dt 
 = \frac{2^{\alpha+\beta+1}}{2n+\alpha+\beta+1} \frac{\Gamma(n+\alpha+1)\,\Gamma(n+\beta+1)}{n!\,\Gamma(n+\alpha+\beta+1)},
\end{align}
in particular \[h_0^{(\alpha,\beta)} = 2^{\alpha+\beta+1} 
\dfrac{\Gamma(\alpha+1)\,\Gamma(\beta+1)}{\Gamma(\alpha+\beta+2)}.\]

The polynomial $P_n^{(\alpha,\beta)}(t)$ is of degree $n$, and its leading coefficient $k_n^{(\alpha,\beta)}$
is given by 
\begin{equation}\label{leadingcoef}
k_n^{(\alpha,\beta)} = \frac{1}{2^n}\, \binom{2n + \alpha + \beta}{n}, \quad n \in \mathbb{N}_0.
\end{equation}
Jacobi polynomials satisfy a three-term recurrence relation
\begin{align}
2(n+1) &(n + \alpha+\beta+1) (2n + \alpha+\beta) P_{n+1}^{(\alpha,\beta)}(t) = \nonumber\\
= & (2n + \alpha+\beta+1) \left\{(2n + \alpha+\beta+2)(2n + \alpha+\beta) t + \alpha^2-\beta^2\right\}P_n^{(\alpha,\beta)}(t) \label{three-term}\\
  &- 2(n+\alpha) (n + \beta) (2n + \alpha+\beta+2) P_{n-1}^{(\alpha,\beta)}(t). \nonumber
\end{align}
The derivative of a Jacobi polynomial is again a Jacobi polynomial, 
\begin{equation}\label{derJ}
\frac{d}{d t} P_n^{(\alpha,\beta)}(t) = \frac{n+\alpha + \beta+1}{2} P_{n-1}^{(\alpha+1,\beta+1)}(t).
\end{equation}

The following relations between different families of the Jacobi polynomials can be found in 
(\cite[Chapt. 22]{AS}):
\begin{eqnarray}
&~& (1-t) P_n^{(\alpha+1,\beta)}(t) = \frac{2(n+\alpha+1)}{2n+\alpha+\beta+2}P_n^{(\alpha,\beta)}(t) - \frac{2(n+1)}{2n+\alpha+\beta+2}P_{n+1}^{(\alpha,\beta)}(t),\label{RAF2}\\
&~& P_n^{(\alpha,\beta)}(t) = a_n^{(\alpha,\beta)} P_n^{(\alpha+1,\beta)}(t) - b_n^{(\alpha,\beta)}P_{n-1}^{(\alpha+1,\beta)}(t),\label{RAF}\\
&~& P_n^{(\alpha,\beta)}(t) = a_n^{(\alpha,\beta)} P_n^{(\alpha,\beta+1)}(t) + b_n^{(\beta,\alpha)}P_{n-1}^{(\alpha,\beta+1)}(t),\label{RAF0}
\end{eqnarray}
where
\begin{equation}\label{coef-a-b}
a_n^{(\alpha,\beta)} = \frac{n+\alpha+\beta+1}{2\,n + \alpha + \beta+1}, \quad b_n^{(\alpha,\beta)} = \frac{n+\beta}{2\,n + \alpha + \beta+1}.
\end{equation}
They are useful to prove the following two relations that we need later. 

\begin{lem} \label{lem:Jacobi-recur}
For $\alpha > -1$ and $\beta > 0$, 
\begin{align}
(1-t)\,\frac{d}{d t} P_n^{(\alpha,\beta)}(t) = \alpha \, P_{n-1}^{(\alpha,\beta+1)}(t) - n \,P_n^{(\alpha-1,\beta+1)}(t),
   \label{jacobi-1} \\
-\alpha P_{n}^{(\alpha,\beta)}(t) + (1-t)\,\frac{d}{d t} P_n^{(\alpha,\beta)}(t) = -(n+\alpha) \,P_n^{(\alpha-1,\beta+1)}(t).
\label{jacobi-2}
\end{align}
\end{lem}

\begin{proof}
From \eqref{derJ} and \eqref{RAF2}, we deduce
\begin{align*}
(1-t)\,\frac{d}{d t} P_n^{(\alpha,\beta)}(t) =\ & (1-t)\,\frac{n+\alpha+\beta +1}{2}P_{n-1}^{(\alpha+1,\beta+1)}(t)\\
=\ & \frac{n+\alpha+\beta+1}{2n+\alpha+\beta+1} \left[(n+\alpha)\,P_{n-1}^{(\alpha,\beta+1)}(t) - n\,P_n^{(\alpha,\beta+1)}(t)\right].
\end{align*}
Replacing the last term by 
\[
P_{n}^{(\alpha,\beta+1)}(t) = \frac{2n+\alpha+ \beta +1}{n+\alpha + \beta +1}P_{n}^{(\alpha-1,\beta+1)}(t)+
\frac{n+\beta + 1}{n+\alpha + \beta +1}P_{n-1}^{(\alpha,\beta+1)}(t),
\]
according to \eqref{RAF}, and simplifying the result, we obtain \eqref{jacobi-1}. The relation \eqref{jacobi-2} is deduced from 
\eqref{jacobi-1} upon using $P_n^{(\alpha,\beta-1)}(t) - P_n^{(\alpha-1,\beta)}(t) = P_{n-1}^{(\alpha,\beta)}(t)$,
which follows from \eqref{RAF} and \eqref{RAF0}.
\end{proof}

\subsection{Orthogonal polynomials on the unit ball}
For $x,y  \in \mathbb{R}^d$, we use the usual notation of $\|x\|$ and $\langle x,y \rangle $ to denote the Euclidean 
norm of $x$ and the dot product of $x,y$.  The unit ball and the unit sphere in $\mathbb{R}^d$ are 
{denoted}, respectively, by
$$
\mathbb{B}^d :=\{x\in \mathbb{R}^d: \|x\| \le 1\} \qquad \textrm{and} \qquad \mathbb{S}^{d-1}:=\{\xi\in \mathbb{R}^d: \|\xi\| = 1\}.
$$
For $\mu \in \mathbb{R}$, let $W_\mu$ be the weight function defined by
$$
    W_\mu(x) = (1-\|x\|^2)^\mu, \qquad  \|x\| < 1.
$$
The function $W_\mu$ is integrable on the unit {ball} if $\mu > -1$, for which we denote the normalization 
constant of $W_\mu$ by ${b_\mu}$, 
\[
b_\mu := \left(\int_{{\mathbb{B}^d}}\, W_\mu(x) \, dx\right)^{-1} = \frac{\Gamma(\mu + d/2 + 1)}{\pi^{d/2}\Gamma(\mu+1)}.
\]
The classical orthogonal polynomials on the unit ball are orthogonal with respect to the inner product
\begin{equation}\label{ball-ip}
   \langle f, g \rangle _\mu = b_\mu \int_{{\mathbb{B}^d}}\, f(x)\, g(x) \, W_\mu(x) \, dx,
\end{equation}
which is normalized in such a way that $\langle 1,1\rangle_\mu = 1$. 

Let $\Pi^d$ denote the space of polynomials in $d$ real variables. For $n = 0,1,2,\ldots,$
let $\Pi_n^d$ denote the linear space of polynomials in several variables of (total) degree 
at most $n$ and let $\mathcal{P}_n^d$ denote the space of homogeneous polynomials of degree 
$n$. It is well known that 
\[
  \dim \Pi_n^d = \binom{n+d}{n} \quad \hbox{and} \quad \dim \mathcal{P}_n^d = \binom{n+d-1}{n}:= r_n^d.
\]
A polynomial $P \in \Pi_n^d$ is called orthogonal with respect to $W_\mu$ on the ball if 
$\langle P, Q\rangle_\mu =0$ for all $Q \in \Pi_{n-1}^d$, that is, if it is orthogonal to all polynomials of lower degree. Let $\mathcal{V}_n^d(W_\mu)$ denote the space of orthogonal polynomials of total degree $n$ with respect to \(W_\mu\). Then $\dim \, \mathcal{V}_n^d(W_\mu) = r_n^d.$

For $n\ge 0$, let $\{P^n_{\alpha}(x) : |\alpha|=n\}$ denote a basis of $\mathcal{V}_n^d(W_\mu)$. 
Then, every element of $\mathcal{V}_n^d(W_\mu)$ is orthogonal to the polynomials of lower degree. If the 
elements of the basis are also orthogonal to each other, that is, $\langle  P_\alpha^n, P_\beta^n \rangle _\mu=0$ whenever \(\alpha \ne \beta\), we call the basis mutually orthogonal. If, in addition, $\langle P_\alpha^n, P_\alpha^n \rangle _\mu =1$, we call the basis orthonormal. 

Harmonic polynomials of $d$-variables are homogeneous polynomials in $\mathcal{P}_n^d$ that satisfy the Laplace equation $\Delta Y = 0$, where 
\(\Delta = \frac{\partial^2}{\partial x_1^2} + \ldots + \frac{\partial^2}{\partial x_d^2}\) is the usual Laplace operator. Let $\mathcal{H}_n^d$ denote the space of harmonic polynomials of degree $n$. It is well known that
$$
         a_n^d: = \dim \mathcal{H}_n^d = \binom{n+d-1}{n} - \binom{n+d-3}{d-1}.
$$
Spherical harmonics are the restriction of harmonic polynomials on the unit sphere. If $Y \in \mathcal{H}_n^d$, {then} $Y(x) = r^n Y(\xi)$ in spherical–polar {coordinates} $x = r \xi$, so that 
$Y$ is uniquely determined by its restriction on the sphere. We shall also use $\mathcal{H}_n^d$ to denote the space of spherical harmonics of degree $n$. Let $d \sigma$ denote the surface measure
and $\omega_d$ denote the surface area, 
$$
  \omega_d := \int_{\mathbb{S}^{d-1}} d\sigma = \frac{2\, \pi^{d/2}}{\Gamma(d/2)}.
$$
Spherical harmonics of different degrees are orthogonal with respect to the inner product
\[
   \langle f, g \rangle_{\mathbb{S}^{d-1}}: = \frac{1}{\omega_d} \int_{\mathbb{S}^{d-1}} f(\xi) g(\xi) d\sigma(\xi). 
\]
Moreover, by Euler's equation for homogeneous polynomials, 
\[x \cdot \nabla Y_\nu^n(x) = n Y_\nu^n(x).
\]

In spherical-polar coordinates $x = r \xi$, $ r > 0$ and $\xi \in \mathbb{S}^{d-1}$, a mutually orthogonal basis of $\mathcal{V}_n^d(W_\mu)$ can be given in terms of the Jacobi polynomials stated in Sec.2.1 and spherical harmonics. Let $P_n^{(\alpha,\beta)}(t)$ denote the usual Jacobi polynomial of degree $n$.
The following well-known result provides such an orthogonal basis of $\mathcal{V}_n^d(W_\mu)$ (see, for example, \cite{DX}). 

\begin{thm}
For $n \in \mathbb{N}_0$ and $0 \le j \le n/2$, let $\{Y_\nu^{n-2j}: 1\le \nu\le a_{n-2j}^d\}$ denote an orthonormal basis for $\mathcal{H}_{n-2j}^d$. Define 
\begin{equation}\label{baseP}
P_{j,\nu}^{n,\mu}(x) = P_{j}^{(\mu, n-2j + \frac{d-2}{2})}(2\,\|x\|^2 -1)\, Y_\nu^{n-2j}(x). 
\end{equation}
Then, the set $\{P_{j,\nu}^{n,\mu}(x): 0 \le j \le \lfloor n/2 \rfloor, \,1 \le \nu \le a_{n-2j}^d \}$ is a mutually
orthogonal basis of $\mathcal{V}_n^d(W_\mu)$. More precisely, 
$$
\langle P_{j,\nu}^{n,\mu}(x), P_{k,\eta}^{m,\mu}(x)\rangle_\mu =  H_{j,n}^{\mu}  \delta_{n,m}\,\delta_{j,k}\,\delta_{\nu,\eta},
$$
where $H_{j,n}^{\mu}$ is given by 
\begin{equation} \label{eq:Hjn-mu}
 H_{j,n}^{\mu}: = \frac{(\mu +1)_j (\frac{d}{2})_{n-j} (n-j+\mu+ \frac{d}{2})}
    { j! (\mu+\frac{d+2}{2})_{n-j} (n+\mu+ \frac{d}{2})}, 
\end{equation} 
{where $(a)_n= a(a+1) \ldots (a+n-1)$ denotes the Pochhammer symbol.} 
\end{thm}

\section{Sobolev orthogonal polynomials on the unit ball}
\setcounter{equation}{0}

Let us consider the \textbf{Sobolev} inner product 

\begin{align*}
\langle f,g\rangle_{\nabla,\mathbb{B}^d} &:=  \lambda\, \int_{\mathbb{B}^d} f(x) g(x) (1-\|x\|^2)^{\kappa} dx +
    \int_{\mathbb{B}^d} \nabla f(x)\,\cdot \nabla g(x) dx,
\end{align*}
where \(\lambda > 0\). And let $\{Y_\nu^{n-2j}: 1 \le \nu \le a_{n-2j}^d\}$ be an orthonormal basis of $\mathcal{H}_{n-2j}^d$.
To construct an orthogonal basis for the linear space \(\mathcal{R}_n^d\) defined in \eqref{space_Rn}, we can try polynomials of the form 
\[
   R_{j,\nu}^{n}(x) := q_{j}(2 \|x\|^2 -1) Y_\nu^{n-2j}(x),
\]
where \(q_{j}\) is a polynomial of degree j in one variable satisfying the orthogonality condition
\[
\langle (1-\|x\|^2) R_{j,\nu}^{n}(x),(1-\|x\|^2)R_{k,\eta}^{m}(x)\rangle_{\nabla,\mathbb{B}^d} = K_{j,n}  \delta_{n,m}\,\delta_{j,k}\,\delta_{\nu,\eta}.
\]

\begin{prop} \label{prop-3.1}
Let $\{Y_\nu^{n-2j}: 1 \le \nu \le a_{n-2j}^d\}$ be an orthonormal basis of $\mathcal{H}_{n-2j}^d$ and let
\[
R_{j,\nu}^{n}(x) := q_{j}(2 \|x\|^2 -1) Y_\nu^{n-2j}(x)
\]
be a basis of Sobolev mutually orthogonal polynomials then $q_{j} = q_{j}^{(\beta)}$ is orthogonal with respect to the univariate \textbf{Sobolev inner product} 
\begin{align}
    (f,g)_{\beta} := & \frac{\lambda}{2^{\kappa+3}} \int_{-1}^1 q_{j}(t) q_{k}(t) (1-t)^{\kappa + 2} (1+t)^{\beta} dt \nonumber\\
 & +    \int_{-1}^1 \frac{d}{dt}\left((1-t)q_j(t)\right) \frac{d}{dt}\left((1-t)q_k(t)\right) (1+t)^{\beta+1} dt, \label{sob_ip_beta} 
\end{align}
where \(\beta = n -2j + \frac{d-2}{2}\). 
\end{prop}

\begin{proof}
From \eqref{Sob_ip} we get
\begin{align}
\langle R_{j,\nu}^{n}(x),R_{k,\eta}^{m}(x)\rangle_{S} =& \lambda\, \int_{\mathbb{B}^d} R_{j,\nu}^{n}(x) R_{k,\eta}^{m}(x) (1-\|x\|^2)^{\kappa + 2} dx \nonumber \\
   &+ \int_{\mathbb{B}^d} \nabla ((1-\|x\|^2)R_{j,\nu}^{n}(x))\,\cdot \nabla ((1-\|x\|^2)R_{k,\eta}^{m}(x)) dx. \label{Sob_ip_mod}
\end{align}
The proof uses the following identity: 
\begin{equation}\label{changevar}
        \int_{\mathbb{B}^d} f(x) dx = \int_0^1 r^{d-1}\int_{\mathbb{S}^{d-1}}f(r\,\xi)\, d\sigma (\xi)\,dr
\end{equation}
that arises from the spherical-polar coordinates $x=r\,\xi$, $\xi\in \mathbb{S}^{d-1}$. In this way, the first integral in \eqref{Sob_ip_mod} gives
\begin{align*}
I_1 =& \int_{\mathbb{B}^d} R_{j,\nu}^{n}(x) R_{k,\eta}^{m}(x) (1-\|x\|^2)^{\kappa+2} dx =\\
   =& \int_{0}^1 q_{j}(2 r^2 -1) q_{k}(2 r^2 -1) (1-r^2)^{\kappa + 2} r^{n-2j+m-2k+d-1} dr \times \\
   &\times \int_{\mathbb{S}^{d-1}}  Y_\nu^{n-2j}(\xi) Y_\nu^{m-2k}(\xi) d\sigma(\xi) \quad \text{ (with } n = m, j = k) \\
=& \int_{0}^1 q_{j}(2 r^2 -1) q_{k}(2 r^2 -1) (1-r^2)^{\kappa + 2} r^{2n-4j+d-1} dr\\ 
& \times \omega_d \delta_{n-2j,m-2k} \delta_{\nu,\eta} \\
=& \int_{0}^1 q_{j}(t) q_{k}(t) \left(\frac{1-t}{2}\right)^{\kappa + 2} \left(\frac{1+t}{2}\right)^{\frac{2n-4j+d-2}{2}}\frac{1}{4} dt\\ 
& \times \omega_d \delta_{n-2j,m-2k} \delta_{\nu,\eta} \\
=& \frac{1}{2^{\kappa+3+n-2j+d/2}} \int_{-1}^1 q_{j}(t) q_{k}(t) (1-t)^{\kappa + 2} (1+t)^{n-2j+\frac{d}{2}-1} dt\\
&\times \omega_d\delta_{n-2j,m-2k} \delta_{\nu,\eta},
\end{align*}
after the change of variable \(t = 2r^2 - 1\).

For the second integral, integration by parts gives
\begin{align*}
I_2 =& \int_{\mathbb{B}^d} \nabla ((1-\|x\|^2)R_{j,\nu}^{n}(x))\,\cdot \nabla ((1-\|x\|^2)R_{k,\eta}^{m}(x)) dx \\
=& - \int_{\mathbb{B}^d} \Delta ((1-\|x\|^2)R_{j,\nu}^{n}(x))(1-\|x\|^2)R_{k,\eta}^{m}(x) dx \\
& + \int_{\mathbb{B}^d} \dive \left(\nabla((1-\|x\|^2)R_{j,\nu}^{n}(x))(1-\|x\|^2)R_{k,\eta}^{m}(x)\right) dx =\\
=& - \int_{\mathbb{B}^d} \Delta ((1-\|x\|^2)R_{j,\nu}^{n}(x))(1-\|x\|^2)R_{k,\eta}^{m}(x) dx,
\end{align*}
which follows from the divergence theorem and the fact that \(1-\|x\|^2\) vanishes on the sphere.

Now, let us write
\begin{align*}
f_j(\|x\|^2) &:= (1-\|x\|^2)  q_{j}(2 r^2 -1)  
\end{align*}
in this way 
\begin{align*}
\Delta &((1-\|x\|^2)R_{j,\nu}^{n}(x)) = \Delta \left( f_j(\|x\|^2) Y_\nu^{n-2j}(x)\right) = \sum_{i=1}^d \partial_i^2 \left( f_j(\|x\|^2) Y_\nu^{n-2j}(x)\right)\\
&= \Delta (f_j(\|x\|^2)) Y_\nu^{n-2j}(x)+ 4 f_j'(\|x\|^2) \langle  x,\nabla\rangle Y_\nu^{n-2j}(x)
+ f_j(\|x\|^2) \Delta Y_\nu^{n-2j}(x) \\
& = \left[ 4 \|x\|^2 f_j''(\|x\|^2)) + 2(2n-4j + d) f_j'(\|x\|^2)\right] Y_\nu^{n-2j}(x),
\end{align*}
since  
\begin{align*}\langle x,\nabla\rangle Y_\nu^{n-2j}(x) &= (n-2j) Y_\nu^{n-2j}(x),\\
\Delta Y_\nu^{n-2j}(x) &= 0,\\
\Delta (f_j(\|x\|^2)) &= 2 d f_j'(\|x\|^2) + 4 \|x\|^2 f_j''(\|x\|^2).
\end{align*}

Thus, using polar coordinates $ x = r \xi$, for the second integral we obtain 
\begin{align*}
I_2 =& - \int_{\mathbb{B}^d} \left[ 4 \|x\|^2 f_j''(\|x\|^2)) + 2(2n-4j + d) f_j'(\|x\|^2)\right] f_k(\|x\|^2) Y_\nu^{n-2j}(x) Y_\eta^{m-2k}(x)dx\\
=& - \int_0^1 \left[ 4 r^2 f_j''(r^2)) + 2(n-2j + d) f_j'(r^2)\right] f_k(r^2) r^{2n-4j+d-1} dr \\
&\times \omega_d \delta_{n-2j,m-2k} \delta_{\nu,\eta}.
\end{align*}
Next, since \(f_k(r^2) r^{2n-4j+d}\) vanishes at the ends of the interval \([0,1]\), integration by parts gives
\begin{align*}
&\int_0^1 4 r f_j''(r^2) f_k(r^2) r^{2n-4j+d} dr =
- 2 \int_0^1  f_j'(r^2) \frac{d}{dr}\left(f_k(r^2) r^{2n-4j+d}\right) dr \\
&= - 4 \int_0^1  f_j'(r^2) f_k'(r^2) r^{2n-4j+d+1} dr  -
2(2n-4j+d) \int_0^1  f_j'(r^2) f_k(r^2) r^{2n-4j+d-1} dr,
\end{align*}
and we get
\[
I_2 = 4 \int_0^1  f_j'(r^2) f_k'(r^2) r^{2n-4j+d+1} dr \omega_d \delta_{n-2j,m-2k} \delta_{\nu,\eta}
\]
Finally, using the change of variables \( t = 2 r^2 -1\), we deduce
\begin{align*}
 4 \int_0^1  f_j'(r^2)& f_k'(r^2) r^{2n-4j+d+1} dr \\
=& \int_0^1 \frac{d}{dr}f_j(r^2) \frac{d}{dr}f_k(r^2)r^{2n-4j+d-1} dr \\
=& \frac{1}{2^{n-2j+d/2}} \int_{-1}^1 \left(- q_j(t) +  (1-t)q_j'(t)\right) \\
& \times \left(- q_k(t) +  (1-t)q_k'(t)\right) (1+t)^{n-2j+d/2} dt \\
=& \frac{1}{2^{n-2j+d/2}} \int_{-1}^1 \frac{d}{dt}\left((1-t)q_j(t)\right) \frac{d}{dt}\left((1-t)q_k(t)\right) (1+t)^{n-2j+d/2} dt,
\end{align*}
and the result follows.
\end{proof}

\section{The radial parts}
\setcounter{equation}{0}

From now on, we will denote by \(q_{k}^{(\beta)}\) the polynomial of degree \(k\) 
orthogonal with respect to the Sobolev inner product \eqref{sob_ip_beta},
and normalized in such a way that it has the same leading coefficient as the Jacobi polynomial $P_k^{(1,\beta)}(t)$, that is,
\begin{align}\label{qk-defn}
   q_{k}^{(\beta)}(t) = k^{(1, \beta)}_k\, t^k + \ldots, 
      \quad k_k^{(1,\beta)} := \frac{1}{2^k} \binom{2k + \beta + 1}{k}. 
\end{align}
With this normalization, we have $q_{0}^{(\beta)}(t) = P_{0}^{(1,\beta)}(t)= 1$.  
Our next result relates $q_k^{(\beta)}$ with the Jacobi polynomials.

\begin{prop}
Let \(\beta >-1\). Then, for $k\ge 0$, 
\begin{equation} \label{connect_qj}
P_{k}^{(1, \beta)}(t) = q_k^{(\beta)}(t) + \sum_{j=k-(\kappa+1)}^{k-1}  a^{(\beta)}_{k,j} q_j^{(\beta)}(t) \quad \text{with}\quad
    a^{(\beta)}_{k,j} = \frac{(P_{k}^{(1, \beta)}, q_j^{(\beta)})_{\beta} } { (q_{j}^{(\beta)},q_{j}^{(\beta)})_{\beta} }.
\end{equation}
\end{prop}

\begin{proof}
Let us expand $P_k^{(1,\beta)}(t)$ in terms of the sequence of polynomials
\(\{q_j^{(\beta)}(t)\}_{j\geqslant 0}\)
which are orthogonal with respect to \((\cdot,\cdot)_{\beta}\)
\begin{equation} \label{suma_qj}
P_{k}^{(1, \beta)}(t) = \sum_{j=0}^k  a^{(\beta)}_{k,j} q_j^{(\beta)}(t) \quad \text{with}\quad
    a^{(\beta)}_{k,j} = \frac{(P_{k}^{(1, \beta)}, q_j^{(\beta)})_{\beta} } { (q_{j}^{(\beta)},q_{j}^{(\beta)})_{\beta} }.
\end{equation}
Since $q_k^{(\beta)}$ is normalized so that its leading coefficient is the same as the one of $P_k^{(1,{\beta})}$, it follows that  $a^{(\beta)}_{k,k} = 1.$
We now {calculate} $a^{(\beta)}_{k,j}$ for $0 \le j < k$. We need to compute
\begin{align}
(P_{k}^{(1, \beta)}, q_j^{(\beta)})_{\beta} &= \frac{\lambda}{2^{\kappa+3}} \int_{-1}^1
P_{k}^{(1, \beta)}(t)\, q_j^{(\beta)}(t)(1-t)^{\kappa+1} \omega_{1,\beta}(t) dt + \nonumber\\
& + \int_{-1}^1 \frac{d}{dt}\left((1-t)P_{k}^{(1, \beta)}(t)\right)\, \frac{d}{dt}\left((1-t)q_j^{(\beta)}(t)\right) \omega_{0,\beta+1}(t) dt. \label{nun_ajk}
\end{align}
From \eqref{jacobi-2} we have
\begin{equation} \label{d_Pk}
\frac{d}{dt}\left((1-t)P_{k}^{(1, \beta)}(t)\right) = - (k+1)P_{k}^{(0, \beta+1)}(t). 
\end{equation}
Therefore the second summand in \eqref{nun_ajk} is null for \(0\leqslant j <k\), whereas the first integral also vanishes for \(0\leqslant j <k-(\kappa+1)\). In this way expression \eqref{suma_qj} contains only \(\kappa+2\) terms.
\end{proof}

Of course, the connection formula \eqref{connect_qj} can be identified as the Gram--Schmidt algorithm applied to the sequence \(\{P_{k}^{(1, \beta)}(t)\}_{k\geqslant 0}\) to obtain an orthogonal basis with respect to the inner product \((\cdot, \cdot)_{\beta}\), but here, the algorithm uses only the \(\kappa+1\) previous polynomials.

\subsection{An alternative algorithm}

We show that the coefficients in \eqref{connect_qj} can be obtained recursively. The key for this algorithm is a \(2\kappa+3\) term recurrence formula for the Jacobi polynomials \(\{P_{k}^{(1, \beta)}(t)\}_{k\geqslant 0}\).
With \(\alpha = 1\), the three-term recurrence relation \eqref{three-term} can be written as
\begin{equation} \label{three-term-alpha-1}
(1-t)P_k^{(1,\beta)}(t) = a_k P_{k+1}^{(1,\beta)}(t) + b_k P_k^{(1,\beta)}(t) + c_k P_{k-1}^{(1,\beta)}(t),
\end{equation}
where
\begin{align*}
a_k =& -\frac{2(k+1)(k + \beta+2)}{(2k + \beta+2)(2k +\beta+3)},\\
b_k =&  \frac{4(k+1)(k + \beta+1)}{(2k + \beta+1)(2k +\beta+3)}, \\
c_k =&-\frac{2(k+1)(k + \beta)}{(2k + \beta+1)(2k +\beta+2)}.
\end{align*}
Iterating \eqref{three-term-alpha-1} we can obtain a representation for \((1-t)^{\kappa+1}P_k^{(1,\beta)}(t)\) in terms of the Jacobi polynomials as follows:
\begin{equation} \label{coef:gamma}
(1-t)^{\kappa+1}P_k^{(1,\beta)}(t) = \sum_{j=k-\kappa-1}^{k+\kappa+1} \gamma_{k,j}^{(\kappa+1)} P_j^{(1,\beta)}(t),
\end{equation}
where the coefficients \(\gamma_{k,j}^{(\kappa)}\) satisfy the recurrence
\begin{align*}
\gamma_{k,j}^{(h+1)} =&  c_{j+1} \gamma_{k,j+1}^{(h)} + b_j \gamma_{k,j}^{(h)}+ a_{j-1} \gamma_{k,j-1}^{(h)},
\end{align*}
for  \(h = 0, \ldots \kappa\), assuming \(\gamma_{k,j}^{(h)}= 0\) for \(j>k+h+1\) or \(j<k-h-1\), and the initial conditions
\begin{align*}
\gamma_{k,k+1}^{(0)} &= a_k, \quad
\gamma_{k,k}^{(0)} = b_k, \quad
\gamma_{k,k-1}^{(0)} = c_k.
\end{align*}

Let us assume \(k \geqslant \kappa+1\). Next, we compute the numerator of \(a^{(\beta)}_{k,j}\) for \(k-\kappa-1 \leqslant j < k\) in \eqref{suma_qj}. 

\begin{prop}
Let us denote
\begin{align*}
 A_k &= ((P_{k}^{(1, \beta)}, q_{k-\kappa-1}^{(\beta)})_{\beta}, \ldots, (P_{k}^{(1, \beta)}, q_{k-1}^{(\beta)})_{\beta})^t,\\
 \Gamma_k &= (\gamma_{k,k-\kappa-1}^{(\kappa+1)}, \ldots, \gamma_{k,k-1}^{(\kappa+1)})^t,
\end{align*}
then 
\[
A_k = C_k \Gamma_k
\]
where \(C_k\) is a lower banded triangular matrix with entries 
\[
c_{i,j} = \frac{\lambda}{2^{\kappa+3}} \int_{-1}^1
q_i^{(\beta)}(t) \, P_{j}^{(1, \beta)}(t)\,
 \omega_{1,\beta}(t) dt,
\] 
for \(i=k-\kappa-1,\ldots,k-1\) and \(j=k-\kappa-1,\ldots,i\).
\end{prop}
\begin{proof}
From \eqref{d_Pk}, we obtain
\begin{align}
(P_{k}^{(1, \beta)}, q_i^{(\beta)})_{\beta} &= \frac{\lambda}{2^{\kappa+3}} \int_{-1}^1
P_{k}^{(1, \beta)}(t)\, q_i^{(\beta)}(t)(1-t)^{\kappa+1}
 \omega_{1,\beta}(t) dt \nonumber\\
&= \sum_{j=k-\kappa-1}^{k+\kappa+1} \gamma_{k,j}^{(\kappa+1)} \frac{\lambda}{2^{\kappa+3}} \int_{-1}^1 
q_i^{(\beta)}(t) \, P_{j}^{(1, \beta)}(t)\,
 \omega_{1,\beta}(t) dt \label{Pk_qi}\\
 &= \sum_{j=k-\kappa-1}^{i} \gamma_{k,j}^{(\kappa+1)} c_{i,j},\nonumber
\end{align}
where
\[
c_{i,j} = \frac{\lambda}{2^{\kappa+3}} \int_{-1}^1
q_i^{(\beta)}(t) \, P_{j}^{(1, \beta)}(t)\,
 \omega_{1,\beta}(t) dt.
\]
Observe that \(c_{i,j} =0\) if \(i<j\).
\end{proof}

To compute recursively \(A_k\) we don't need to compute all the integrals in \(C_k\) since the first \(\kappa\) rows in the matrix are known from the previous iterations and the last row 
can be obtained from \eqref{connect_qj} with \(k\) replaced by \(k-1\). In fact, for \(k-\kappa-1\leqslant j<k-1\) we have
\begin{align*}
c_{k-1,j} &= \frac{\lambda}{2^{\kappa+3}} \int_{-1}^1
q_{k-1}^{(\beta)}(t) \, P_{j}^{(1, \beta)}(t)\,
 \omega_{1,\beta}(t) dt \\
 &=\frac{\lambda}{2^{\kappa+3}} \int_{-1}^1
 P_{j}^{(1, \beta)}(t)\left( P_{k-1}^{(1, \beta)}(t) - \sum_{h=k-\kappa-2}^{k-2}  a^{(\beta)}_{k-1,h} q_h^{(\beta)}(t)\right)\,
 \omega_{1,\beta}(t) dt \\
&= - \sum_{h=k-\kappa-2}^{k-2} a^{(\beta)}_{k-1,h} c_{h,j}.
\end{align*}

The diagonal entries in \(C_k\) satisfy
\[
c_{i,i} = \frac{\lambda}{2^{\kappa+3}} \int_{-1}^1
q_i^{(\beta)}(t) \, P_{i}^{(1, \beta)}(t)\,
 \omega_{1,\beta}(t) dt = \frac{\lambda}{2^{\kappa+3}} h_{i}^{(1, \beta)}.
\]
\medskip

Finally, the denominator of \(a^{(\beta)}_{k,j}\) for $k-\kappa-1 \leqslant j < k$ in \eqref{suma_qj} can be obtained recursively. For every value \(k\) we have
\begin{align}
(q_k^{(\beta)}, q_k^{(\beta)})_{\beta} &= (P_{k}^{(1, \beta)}, q_k^{(\beta)})_{\beta} = \frac{\lambda}{2^{\kappa+3}} \int_{-1}^1
P_{k}^{(1, \beta)}(t)\, q_k^{(\beta)}(t)(1-t)^{\kappa+1} \omega_{1,\beta}(t) dt + \nonumber\\
& + \int_{-1}^1 \frac{d}{dt}\left((1-t)P_{k}^{(1, \beta)}(t)\right)\, \frac{d}{dt}\left((1-t)q_k^{(\beta)}(t)\right) \omega_{0,\beta+1}(t) dt \label{den_akk}
\end{align}
Since \(P_{k}^{(1, \beta)}(t)\) and \(q_k^{(\beta)}(t)\) have the same leading coefficient, from \eqref{jacobi-2} we deduce 
\[
\int_{-1}^1 \frac{d}{dt}\left((1-t)P_{k}^{(1, \beta)}(t)\right)\, \frac{d}{dt}\left((1-t)q_k^{(\beta)}(t)\right) \omega_{0,\beta+1}(t) dt = (k+1)^2 h_{k}^{(0, \beta+1)}.
\]
Now, for the first integral in \eqref{den_akk} it holds
\begin{align}
\frac{\lambda}{2^{\kappa+3}} & \int_{-1}^1
P_{k}^{(1, \beta)}(t)\, q_k^{(\beta)}(t)(1-t)^{\kappa+1} \omega_{1,\beta}(t) dt \nonumber\\
&= \frac{\lambda}{2^{\kappa+3}} \int_{-1}^1 \left(
\sum_{j=k-\kappa-1}^{k+\kappa+1} \gamma_{k,j}^{(\kappa+1)} P_j^{(1,\beta)}(t)\right)\, q_k^{(\beta)}(t) \omega_{1,\beta}(t) dt \label{prim_int}\\
& = \sum_{j=k-\kappa-1}^{k} \gamma_{k,j}^{(\kappa+1)} c_{k,j},\nonumber
\end{align}
and therefore we conclude
\begin{align*}
(q_k^{(\beta)}, q_k^{(\beta)})_{\beta} &= \sum_{j=k-\kappa-1}^{k} \gamma_{k,j}^{(\kappa+1)} c_{k,j} + (k+1)^2 h_{k}^{(0, \beta+1)}. 
\end{align*}

\medskip

The case \(0<k<\kappa+1\) follows exactly in the same way, just taking into account that \(P_{j}^{(1, \beta)}(t) = q_j^{(\beta)}(t) = 0\) for \(j<0\) and  all the sums in the previous reasoning starts at \(j = 0\).

\subsection{The case \(\kappa = 0\)}

In the case \(\kappa = 0\) relation \eqref{connect_qj} contains only two terms and we can give a closed representation for the coefficients.

\begin{prop}
Let \(\kappa = 0\) and \(\beta >-1\). Then, for $k\ge 0$, 
\begin{equation}\label{1-2-one-v}
P_k^{(1, \beta)}(t) = q_k^{(\beta)}(t) + d_{k}^{(\beta)}q_{k-1}^{(\beta)}(t),
\end{equation}
where $d_k^{(\beta)}$ is defined by means
\begin{equation}\label{def-d}
d_{0}^{(\beta)}=0, \quad d_{k}^{(\beta)} = - \frac{\lambda k (k + \beta+1)}{4(2 k + \beta) ( 2 k + \beta+1)}\frac{h_{k}^{(1,\beta)}}{\widetilde h_{k-1}^{(\beta)}},\quad k = 1,2, \ldots, 
\end{equation}
with \(\widetilde h_{k-1}^{(\beta)} = (q_{k-1}^{(\beta)},q_{k-1}^{(\beta)})_{\beta}\). 
\end{prop}

\begin{proof}
In this case relation \eqref{connect_qj} contains only two terms and reduces to 
\[
P_{k}^{(1, \beta)}(t) = \sum_{j=k-1}^k  a^{(\beta)}_{k,j} q_j^{(\beta)}(t) \quad \text{with}\quad
    a^{(\beta)}_{k,j} = \frac{(P_{k}^{(1, \beta)}, q_j^{(\beta)})_{\beta} } { (q_{j}^{(\beta)},q_{j}^{(\beta)})_{\beta} }.
\]
Again, since $q_k^{(\beta)}(t) = k_k^{(1,{\beta})}t^k + \ldots$, it follows that  $a^{(\beta)}_{k,k} = 1.$
To compute $a^{(\beta)}_{k,k-1}$ we proceed as before
\begin{align*}
(P_{k}^{(1, \beta)}, q_{k-1}^{(\beta)})_{\beta} &= \frac{\lambda}{8} \int_{-1}^1
P_{k}^{(1, \beta)}(t)\, q_{k-1}^{(\beta)}(t)(1-t) \omega_{1,\beta}(t) dt + \nonumber\\
& + \int_{-1}^1 \frac{d}{dt}\left((1-t)P_{k}^{(1, \beta)}(t)\right)\, \frac{d}{dt}\left((1-t)q_{k-1}^{(\beta)}(t)\right) \omega_{0,\beta+1}(t) dt, 
\end{align*}
where the second integral vanishes by virtue of \eqref{d_Pk}. Let us denote \(d_{k}^{(\beta)}:= a^{(\beta)}_{k,k-1}\). Then expression \eqref{def-d} follows by 
\begin{align*}
(P_{k}^{(1, \beta)}, q_{k-1}^{(\beta)})_{\beta} &= \frac{\lambda}{8} \int_{-1}^1
P_{k}^{(1, \beta)}(t)\, q_{k-1}^{(\beta)}(t)(1-t) \omega_{1,\beta}(t) dt   = - \frac{\lambda}{8} \frac{k_{k-1}^{(1,\beta)}}{k_{k}^{(1,\beta)}}h_{k}^{(1,\beta)}
\end{align*}
\end{proof}

Our next proposition shows that the coefficients  $d_k^{(\beta)}$ satisfy a nonlinear recursive relation. 

\begin{prop}
The coefficients $d_{k}^{(\beta)}$ satisfy the following recurrence relation:
\begin{align}
& d_1^{(\beta)} = \frac{(P_{1}^{(1, \beta)}, P_{0}^{(1,\beta)})_{\beta}}{(P_{0}^{(0, \beta)}, P_{0}^{(1,\beta)})_{\beta}},\label{d-cero}\\
& d_{k}^{(\beta)} = \frac{(P_{k}^{(1, \beta)}, P_{k-1}^{(1,\beta)})_{\beta}}{(P_{k-1}^{(1, \beta)}, P_{k-1}^{(1,\beta)})_{\beta} - d_{k-1}^{(\beta)} \, (P_{k-1}^{(1, \beta)}, P_{k-2}^{(1,\beta)})_{\beta}},\quad k\ge 2,\label{d-rec}
\end{align}
\end{prop}

\begin{proof}
Relation \eqref{d-cero} follows from \eqref{1-2-one-v} for \(k = 1\), and
\(q_0^{(\beta)}(t) = P_0^{(1, \beta)}(t) = 1\).
Next, iterate identity \eqref{1-2-one-v} to obtain
\begin{align*}
q_k^{(\beta)}(t) &= P_k^{(1, \beta)}(t) - d_{k}^{(\beta)}q_{k-1}^{(\beta)}(t)\\
& =  P_k^{(1, \beta)}(t) - d_{k}^{(\beta)}P_{k-1}^{(1, \beta)}(t) + d_{k}^{(\beta)} d_{k-1}^{(\beta)} q_{k-2}^{(\beta)}(t),
\end{align*}
and \eqref{d-rec} follows from \((q_{k}^{(\beta)}, P_{k-1}^{(1,\beta)})_{\beta}= 0\) and considering that \[(q_{k-2}^{(1, \beta)}, P_{k-1}^{(1,\beta)})_{\beta} = (P_{k-2}^{(1, \beta)}, P_{k-1}^{(1,\beta)})_{\beta}.\]
\end{proof}

\bigskip

The inner products in \eqref{d-rec} can be explicitly computed in terms of the parameters \(k_k^{(1,\beta)}\) and \(h_k^{(1,\beta)}\). In fact, using \eqref{d_Pk} we have
\begin{align*}
(P_{k}^{(1, \beta)}, P_{k-1}^{(1,\beta)})_{\beta} &= 
 \frac{\lambda}{8} \int_{-1}^1 P_{k}^{(1, \beta)}(t)\, P_{k-1}^{(1,\beta)}(t) (1-t) \omega_{1,\beta}(t) dt \\
&= - \frac{\lambda}{8} \frac{k_{k-1}^{(1,\beta)}}{k_{k}^{(1,\beta)}} h_{k}^{(1,\beta)} = - \frac{\lambda}{8} \frac{ k (k+1) 2^{\beta+3}}{(2k+\beta)(2k+\beta+1)(2k+\beta+2)}.
\end{align*} 
On the other hand
\begin{align*}
(P_{k-1}^{(1, \beta)}, P_{k-1}^{(1,\beta)})_{\beta} &= 
 \frac{\lambda}{8} \int_{-1}^1 \left(P_{k-1}^{(1, \beta)}(t)\right)^2 (1-t) \omega_{1,\beta}(t) dt \\
&+ \int_{-1}^1 \left(\frac{d}{dt}\left((1-t)P_{k-1}^{(1, \beta)}(t)\right)\right)^2  \omega_{0,\beta+1}(t) dt.
\end{align*} 

For the first integral in this expression, using \eqref{RAF2} we get
\[
(1-t)P_{k-1}^{(1, \beta)}(t) = \frac{2k}{2k+\beta}\left(P_{k-1}^{(0, \beta)}(t)- P_{k}^{(0, \beta)}(t)\right)
\]
and therefore
\begin{align*}
\int_{-1}^1 \left(P_{k-1}^{(1, \beta)}(t)\right)^2& (1-t) \omega_{1,\beta}(t) dt = \frac{2k}{2k+\beta}\int_{-1}^1 P_{k-1}^{(1, \beta)}(t) \left(P_{k-1}^{(0, \beta)}(t)- P_{k}^{(0, \beta)}(t)\right) \omega_{1,\beta}(t) dt\\
&= \frac{2k}{2k+\beta}\left(\frac{k_{k-1}^{(0,\beta)}}{k_{k-1}^{(1,\beta)}} h_{k-1}^{(1,\beta)} +
\frac{k_{k-1}^{(1,\beta)}}{k_{k}^{(0,\beta)}} h_{k}^{(0,\beta)}\right) \\
& = \frac{2 k^2 2^{\beta+3} }{(2k+\beta-1)(2k+\beta)(2k+\beta+1)}
\end{align*}
For the second integral, again in view of \eqref{def-d}, we have
\begin{align*}
\int_{-1}^1 \left(\frac{d}{dt}\left((1-t)P_{k-1}^{(1, \beta)}(t)\right)\right)^2  \omega_{0,\beta+1}(t) dt &= k^2  
\int_{-1}^1 \left(P_{k-1}^{(0, \beta+1)}(t)\right)^2  \omega_{0,\beta+1}(t) dt \\
&= k^2 h_{k-1}^{(0, \beta+1)} = 
\frac{k^2 2^{\beta+3} }{2(2k+\beta)}
\end{align*}

Therefore, after some obvious cancelations \eqref{d-rec} can be written as
\begin{align*}
d_{k}^{\beta} &= \frac{\frac{ k+1 }{(2k+\beta+1)(2k+\beta+2)}}{-\left(\frac{ 2k }{(2k+\beta-1)(2k+\beta+1)} + \frac{4k}{\lambda}\right) + d_{k-1}^{\beta}\frac{ k-1 }{(2k+\beta-1)(2k+\beta-2)}}\\
&= \frac{\frac{ t_{k+1} }{(2k+\beta+1)}}{-\left(\frac{ 1 }{(2k+\beta-1)} + \frac{ 1 }{(2k+\beta+1)} + \frac{4(2k+\beta)}{\lambda}\right)t_k + d_{k-1}^{\beta}\frac{ t_{k-1} }{(2k+\beta-1)}}
\end{align*}
where \(t_k = \frac{k}{2k+\beta}\).

The recurrence relation \eqref{d-rec} shows that $d_k^{\beta}$ can be expressed as a continuous fraction. In this way, we can express $d_k^{\beta}$ in terms of a rational function of $\lambda$ whose numerator and denominator are, respectively, the $k$th and $k+1$th elements of a sequence of orthogonal polynomials. Consequently, we have
\begin{lem} 
For $k \in \mathbb{N}_0$, define the polynomials $r_k^{(\beta)}(s)$ by $r_{-1}^{(\beta)}(s) = 0, r_0^{(\beta)}(s) =1,$ and the recurrence relation
\begin{align} \label{3term-rj}
 \frac{t_{k+1}}{2k+\beta+1} r_{k+1}^{(\beta)}(s) &= -\left(\frac{ 1 }{(2k+\beta-1)} + \frac{ 1 }{(2k+\beta+1)} + (2k+\beta)\,s\right)t_k  r_{k}^{(\beta)}(s) \nonumber \\
 &+   \frac{ t_{k-1} }{(2k+\beta-1)}r_{k-1}^{(\beta)}(s). 
\end{align}
Then 
\[
d_{k}^{\beta} = \frac{r_{k}^{(\beta)}(\frac{4}{\lambda})}{r_{k+1}^{(\beta)}(\frac{4}{\lambda})}\]
\end{lem}
Therefore, the three-term relation in \eqref{3term-rj} offers an effective way to generate $r_{k}^{(\beta)}$. 

\section{The approximate solution of the BVP}
\setcounter{equation}{0}

Let us consider the Fourier expansion of the solution of the BVP \eqref{pde-kappa}-\eqref{boundary} in terms of the orthogonal system of polynomials $\{(1-\|x\|^2)  R_{j,\nu}^{n}(x)\}$, with \(R_{j,\nu}^{n}(x)\) as defined in Prop.~\ref{prop-3.1}
\begin{equation}\label{Fourier-coef}
u(x) = \sum_{n=0}^{\infty} \sum_{j=0}^{\lfloor \frac{n}{2} \rfloor} \sum_{\nu}
     \widehat{u}_{j,\nu}^{n} (1-\|x\|^2) R_{j,\nu}^{n}(x), 
\end{equation}
with
\[
     \widehat{u}_{j,\nu}^{n} = \frac{\langle u(x),(1-\|x\|^2) R_{j,\nu}^{n}(x)\rangle_{\nabla,\mathbb{B}^d}}{\|(1-\|x\|^2) R_{j,\nu}^n(x)\|_{\nabla,\mathbb{B}^d}^2}.
\]

In this way, the coefficients $\widehat{u}_{j,\nu}^{n}$ satisfy
\begin{align*}
\widehat{u}_{j,\nu}^{n}& \|(1-\|x\|^2) R_{j,\nu}^n(x)\|_{\nabla,\mathbb{B}^d}^2\\
& = \lambda\, \int_{\mathbb{B}^d} u(x) (1-\|x\|^2)^{\kappa+1} R_{j,\nu}^n(x)  dx +\int_{\mathbb{B}^d} \nabla u(x)\,\cdot \nabla \left[(1-\|x\|^2) R_{j,\nu}^n(x)\right] dx\\
& = \int_{\mathbb{B}^d}\left[-\Delta u(x) + \lambda (1-\|x\|^2)^{\kappa} u(x) \right] (1-\|x\|^2) R_{j,\nu}^n(x) dx\\
& = \int_{\mathbb{B}^d}f(x) (1-\|x\|^2) R_{j,\nu}^n(x) dx.
\end{align*}
This identity is where the \textbf{fully diagonalized spectral method for the BVP} cames from.

\bigskip

Using spherical polar coordinates, relation \eqref{connect_qj} provides a connection formula for the classical and Sobolev orthogonal polynomials on the unit ball. In fact, with \(\mu = 1, \beta = n-2k + \frac{d-2}{2}\) in \eqref{baseP}, we have
\begin{align}
P_{k,\nu}^{n,1}(x) & = P_{k}^{(1, \beta)}(2\,\|x\|^2 -1)\, Y_\nu^{n-2k}(x)\nonumber \\
 & = \left[q_k^{(\beta)}(t) + \sum_{j=k-(\kappa+1)}^{k-1}  a^{(\beta)}_{k,j} q_j^{(\beta)}(t)\right]Y_\nu^{n-2k}(x)\nonumber \\
& = R_{k,\nu}^{n}(x) + \sum_{j=k-(\kappa+1)}^{k-1} a^{(\beta)}_{k,j} R_{j,\nu}^{n-2(k-j)}(x)\label{connect_PR}
\end{align}
Let us denote
\[
\widetilde{f}_{j,\nu}^{n} := \widehat{u}_{j,\nu}^{n} \|(1-\|x\|^2) R_{j,\nu}^n(x)\|_{\nabla,\mathbb{B}^d}^2 = \int_{\mathbb{B}^d}f(x) (1-\|x\|^2) R_{j,\nu}^n(x) dx.
\]
If we denote by $\widehat{f}_{j,\nu}^{n,1}$ the  Fourier coefficient of $f$ in the basis of classical
ball polynomials, $P_{j,\nu}^{n,1}$, expression \eqref{connect_PR} can be written as follows
\begin{equation} \label{rec_Fourier}
\widehat{f}_{k,\nu}^{n,1} H_{k,n}^{1}  = \widetilde{f}_{k,\nu}^{n} + \sum_{j=k-(\kappa+1)}^{k-1} a^{(\beta)}_{k,j} \widetilde{f}_{j,\nu}^{n-2(k-j)}.
\end{equation}

Therefore, we obtain a recursive method to obtain the coefficients \(\widetilde{f}_{k,\nu}^{n}\).
The computation of the Fourier coefficients of \(f(x)\) in the \(L^2\) basis \(\widehat{f}_{j,\nu}^{n,1}\) could be accomplished with
fast transforms for orthogonal polynomials in the unit ball, as described in \cite{OX}. Observe that the algorithm devised in Section 4.1 to compute the coefficients \(a^{(\beta)}_{k,j}\) in \eqref{connect_PR} can be used simultaneously to implement \eqref{rec_Fourier}.

\section{A numerical experiment}
\setcounter{equation}{0}

In this section, we analyze a numerical experiment to explore the reliability and accuracy of the Sobolev spectral method for solving Dirichlet boundary problems on the unit ball $\mathbb{B}^d$.

We consider a second-order Dirichlet boundary value problem associated with the non-homogeneous Schr\"odinger equation with potential $V(x) = 1$ on the unit circle, i.e. the case \(\kappa = 0, \lambda = 8, d = 2\) of problem \eqref{pde-kappa}-\eqref{boundary}. 
\begin{align*}
 - &\Delta u(x_1,x_2) + \lambda u(x_1,x_2) = f(x_1,x_2), \quad (x_1,x_2) \in \mathbb{B}^2,\\
& u(x_1,x_2) = 0, \quad (x_1,x_2) \in \mathbb{S}^1.
\end{align*}

The right-hand side of the elliptic partial differential equation is given by
\[f(x_1,x_2)= e^{-x_1-x_2} \left(- 6 x_1^2 - 6 x_2^2 -4 x_1  - 4 x_2 + 10\right),\]
and therefore the solution of the boundary value problem is 
\[u(x_1,x_2) = e^{-x_1-x_2} (1-x_1^2-x_2^2).\]
In Figure~\ref{fig1}, we plot the solution $u(x_1,x_2)$ of the boundary value problem. 

\begin{figure}
\centerline{\includegraphics[scale=0.75]{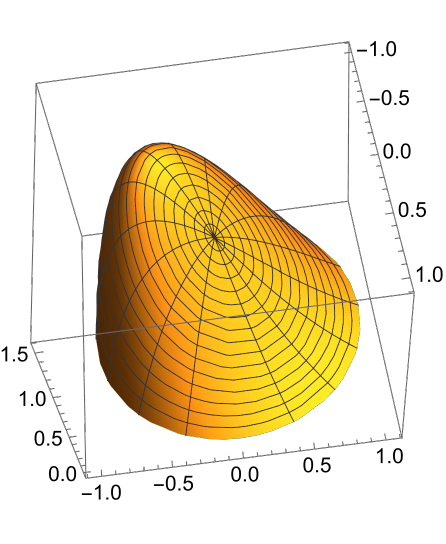}}
\caption{\label{fig1}The solution $u(x_1,x_2)$ of the BVP.}
\end{figure}

\bigskip

For $d =2$,  $\dim \mathcal{H}_n^2 = 2$ for $n \ge 1$. In polar coordinates, $x = (x_1, x_2) = (r \cos \theta, r \sin \theta)$, 
of $\mathbb{R}^2$, an orthogonal basis of $\mathcal{H}_n^2$ is given by
\[
Y^{(1)}_n (x_1,x_2) = r^n \cos n \theta, \quad Y^{(2)}_n (x_1,x_2) = r^n \sin n \theta.
\]

The family of polynomials $(1 - x_1^2 - x_2^2) R^n_{j,\nu}(x_1, x_2)$ orthogonal with respect to the Sobolev inner product
\begin{align*}
\langle f,g\rangle_{\nabla,\mathbb{B}^2} &:=  8\, \int_{\mathbb{B}^2} f(x_1,x_2) g(x_1,x_2)  dx_1 dx_2 +
    \int_{\mathbb{B}^2} \nabla f(x_1,x_2)\,\cdot \nabla g(x_1,x_2) dx_1 dx_2,
\end{align*}
can be easily obtained using polar coordinates. Table~\ref{tab1} contains the first Sobolev polynomials up to a multiplicative constant.

\begin{table}[ht]
\begin{center}
\bgroup
\def\arraystretch{1.5}
\begin{tabular}{|c|c|c|l|l|}
\hline 
\rule[-1ex]{0pt}{2.5ex} $n$ &$j$ & $\nu$ & \text{Polar} & \text{Cartesian} \\ 
\hline 
\rule[-1ex]{0pt}{2.5ex} 0 & 0 & 1 & $q_0^{(0)}(2\rho^2-1)$ & 1 \\ 
\hline 
\rule[-1ex]{0pt}{2.5ex} \multirow{2}{*}{1} & \multirow{2}{*}{0} & 1 & $q_0^{(1)}(2\rho^2-1)\rho \cos{\theta}$ & $x_1$ \\ 
\cline{3-5} 
\rule[-1ex]{0pt}{2.5ex}  &  & 2 & $q_0^{(1)}(2\rho^2-1)\rho \sin{\theta}$ & $x_2$ \\
\hline 
\rule[-1ex]{0pt}{2.5ex} \multirow{3}{*}{2} & \multirow{2}{*}{0} & 1 & $q_0^{(2)}(2\rho^2-1) \rho^2 \cos{2\theta}$ & $x_1^2-x_2^2$ \\ 
\cline{3-5} 
\rule[-1ex]{0pt}{2.5ex}  &  & 2 & $q_0^{(2)}(2\rho^2-1) \rho^2 \sin{2\theta}$ & $x_1 x_2$ \\ 
\cline{2-5} 
\rule[-1ex]{0pt}{2.5ex}  & 1 & 1  & $q_1^{(0)}(2\rho^2-1) $ & $-2 + 7(x_1^2+x_2^2)$ \\
\hline 
\rule[-1ex]{0pt}{2.5ex} \multirow{4}{*}{3} & \multirow{2}{*}{0} & 1 & $q_0^{(3)}(2\rho^2-1) \rho^3 \cos{3\theta}$ & $x_1^3-3 x_1 x_2^2$ \\ 
\cline{3-5} 
\rule[-1ex]{0pt}{2.5ex}  &  & 2 & $q_0^{(3)}(2\rho^2-1) \rho^3 \sin{3\theta}$ & $3 x_1^2 x_2 - x_2^3$ \\ 
\cline{2-5} 
\rule[-1ex]{0pt}{2.5ex}  & \multirow{2}{*}{1} & 1 & $q_1^{(1)}(2\rho^2-1) \rho \cos{\theta}$ & $(-7 + 15 (x_1^2 + x_2^2)) x_1$ \\ 
\cline{3-5} 
\rule[-1ex]{0pt}{2.5ex}  &  & 2 & $q_1^{(1)}(2\rho^2-1) \rho \sin{\theta}$ & $(- 7 + 15 (x_1^2 + x_2^2)) x_2$ \\ 
\hline 
\end{tabular}
\egroup 
\end{center}
\caption{Sobolev polynomials for \(n = 0, 1, 2, 3.\)}
\label{tab1}
\end{table}

Let us denote by
\begin{equation} \label{Fourier-Sob}
u(x_1,x_2) = \sum_{n=0}^{\infty} \sum_{j=0}^{\lfloor \frac{n}{2} \rfloor} \sum_{\nu}
     \widehat{u}_{j,\nu}^{n} (1 - x_1^2 - x_2^2) R_{j,\nu}^{n}(x_1,x_2)
\end{equation}
the Fourier--Sobolev expansion of $u(x_1,x_2)$. The coefficients $\widehat{u}_{j,\nu}^{n}$ satisfy
\begin{align*}
\widehat{u}_{j,\nu}^{n}& \|(1 - x_1^2 - x_2^2) R_{j,\nu}^n(x_1,x_2)\|_{\nabla,\mathbb{B}^2}^2\\
=& 8 \, \int_{\mathbb{B}^2} u(x_1,x_2) (1 - x_1^2 - x_2^2) R_{j,\nu}^n(x_1,x_2)  dx_1 dx_2 \\
& +\int_{\mathbb{B}^2} \nabla u(x_1,x_2)\,\cdot \nabla \left[(1 - x_1^2 - x_2^2) R_{j,\nu}^n(x_1,x_2)\right] dx_1 dx_2\\
=& \int_{\mathbb{B}^2}\left[-\Delta u(x_1,x_2) + 8 u(x_1,x_2) \right] (1 - x_1^2 - x_2^2) R_{j,\nu}^n(x_1,x_2) dx_1 dx_2\\
=& \int_{\mathbb{B}^2}f(x) (1 - x_1^2 - x_2^2) R_{j,\nu}^n(x_1,x_2) dx_1 dx_2,
\end{align*}
and therefore can be easily computed from the function $f(x_1,x_2)$. The partial sums of the series \eqref{Fourier-Sob} provide the approximants of the solution $u(x_1,x_2)$ of the BVP. For instance, the approximant 
\[
u_3(x_1,x_2) = \sum_{n=0}^{3} \sum_{j=0}^{\lfloor \frac{n}{2} \rfloor} \sum_{\nu}
     \widehat{u}_{j,\nu}^{n} (1-x_1^2-x_2^2) R_{j,\nu}^{n}(x_1,x_2)
\]
is given by
\begin{align*}
u_3(x_1,x_2) &= (1-x_1^2-x_2^2) \left(0.9938 - 0.9958 x_1 - 0.9958 x_2 + 0.5505 x_1^2 + 1.1005 x_1 x_2 \right. \\
& \left.+ 0.5505 x_2^2 - 0.1808 x_1^3 - 0.5423 x_1^2 x_2 - 0.5423 x_1 x_2^2 - 0.1808 x_2^3\right).
\end{align*}
This gives a good approximation to the solution $u(x_1,x_2)$ as we can check in Figure~\ref{fig3} which shows a contour plot for the difference $u(x_1,x_2)-u_3(x_1,x_2)$.

\begin{figure}
\centerline{\includegraphics[scale=1]{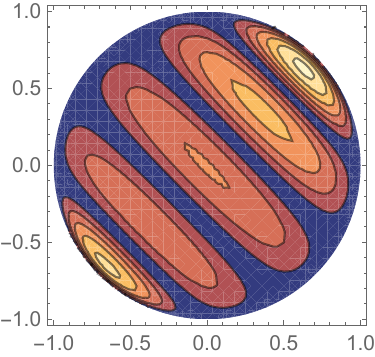}
\includegraphics[scale=0.75]{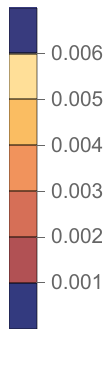}}
\caption{\label{fig3} A contour plot for the difference $u(x_1,x_2)-u_N(x_1,x_2)$ for \(N=3\).}
\end{figure}

To illustrate the convergence of the Sobolev spectral method we have computed the Fourier-Sobolev partial sums
\[
u_N(x_1,x_2) = \sum_{n=0}^{N} \sum_{j=0}^{\lfloor \frac{n}{2} \rfloor} \sum_{\nu}
     \widehat{u}_{j,\nu}^{n} (1-x_1^2-x_2^2) R_{j,\nu}^{n}(x_1,x_2)
\]
for $N = 0, 1, \ldots, 7$. Observe that the computation of $u_N(x_1,x_2)$ implies $(N+1)(N+2)/2$ orthogonal polynomials. Whereas the Sobolev orthogonal polynomials were computed symbolically, the Fourier coefficients were computed numerically using an adaptative cubature procedure on the unit disk.

In Figure~\ref{fig2}, we show a logarithmic plot of the square of the errors in the Sobolev norm
\[ \epsilon_N = \|u(x,y)-u_N(x,y)\|_{\nabla,\mathbb{B}^2}^2,\]
for $N = 0, 1, \ldots, 7$. Clearly, the near straight aligned points indicate an exponential convergence rate. In fact, since the solution is an entire function, the expansion coefficients and error decay super-exponentially.  

\begin{figure}
\centerline{\includegraphics[scale=0.75]{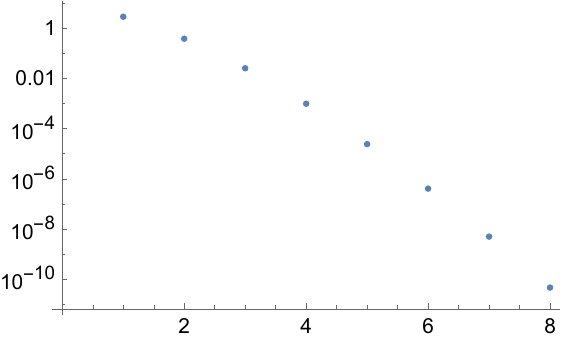}}
\caption{\label{fig2}A logarithmic plot of the errors for $N = 0, 1, \ldots, 7$.}
\end{figure}

\section*{Acknowledgments} 
We thank the anonymous reviewers for their insightful suggestions that improved the clarity and quality of our paper,

\end{document}